\documentclass[final,leqno,onefignum,onetabnum]{siamltex1213}

\usepackage{amscd,amssymb,epsfig,mathtools}

 \newcommand\cof{\operatorname{cof}}

\numberwithin{equation}{section}

\newcommand{\ud}{\,d} 
 
\newcommand{\R}{\mathbb{R}}

\newcommand{\tir}[1]{\ensuremath{\overline {#1}}} 
\newtheorem{thm}{Theorem}[section]

\newtheorem{defn}[thm]{Definition} 
 
\newtheorem{rem}[thm]{Remark}

\def\whsq{\vbox to 5.8pt 
{\offinterlineskip\hrule 
\hbox to 5.8pt{\vrule height 
5.1pt\hss\vrule height 5.1pt}\hrule}}

\def\<{\langle} 
\def\>{\rangle} 
\def\PP{{\mathop{{\rm I}\kern-.2em{\rm P}}\nolimits}} 
\def\FF{{\mathop{{\rm I}\kern-.2em{\rm F}}\nolimits}}   
\def\ZZ{{\mathop{{\rm I}\kern-.2em{\rm Z}}\nolimits}} 
 
\title{
Convergence of a hybrid scheme for the elliptic Monge-Amp\`ere equation}

\author{Gerard Awanou \thanks{Department of Mathematics, Statistics, and Computer Science, M/C 249.
University of Illinois at Chicago, 
Chicago, IL 60607-7045, USA}
(\email{awanou@uic.edu})
}

\begin{document}
\maketitle
\slugger{sinum}{xxxx}{xx}{x}{x--x}

\begin{abstract}
We prove the convergence of a  hybrid discretization to the viscosity solution of the elliptic Monge-Amp\`ere equation. The hybrid discretization uses a standard finite difference discretization in parts of the computational domain where the solution is expected to be smooth and a monotone scheme elsewhere. A motivation for the hybrid discretization is  the lack of an  appropriate Newton solver for the standard finite difference discretization on the whole domain.  
\end{abstract}



\pagestyle{myheadings}
\thispagestyle{plain}
\markboth{Gerard Awanou}{Convergence of a hybrid scheme}

\section{Introduction}
In this paper, we prove the convergence  of a hybrid discretization to the viscosity solution of the elliptic Monge-Amp\`ere equation. The discretization we analyze was proposed by Froese and Oberman in \cite{Oberman2010b}. 
The elliptic Monge-Amp\`ere equation is a fully nonlinear equation, i.e. nonlinear in the highest order derivatives. Unless the domain is smooth and strictly convex and the data are smooth, the solution is not expected to be smooth. 
By elliptic regularity, upon regularization of the data, the solution is smooth on any relatively compact subset. Since computers ''do not see'' the difference between the domain and an arbitrarily close subdomain, several methods provenly convergent for smooth solutions remain effective for non smooth solutions \cite{Awanou-Std-fd-jsc,Awanou-Std04}. Indeed, numerical experiments \cite{Benamou2010,Awanou-Std-fd-jsc} indicate that the discrete equations obtained through standard finite difference discretizations have a discrete convex solution in the sense that a certain discrete Hessian is positive. The solution can be retrieved through appropriate iterative methods. We note that the standard finite difference discretization is commonly used in science and engineering \cite{Headrick05,Chen2010b,Chen2010c}. However it is not known whether 
an appropriate Newton solver can be developed for the standard finite difference discretization. On the other hand numerical experiments reported in \cite{Oberman2010b} indicate that Newton's method can be applied to the nonlinear system resulting from a hybrid discretization. We do not reproduce them in this paper.

The hybrid discretization proposed in \cite{Oberman2010b} uses a  consistent, monotone and stable  scheme in parts of the domain where the solution is not expected to be smooth and a standard discretization in parts of the domain where the solution is smooth. We will often refer to a  consistent, monotone and stable  scheme simply as a monotone scheme. The monotone discretization is known to converge to the viscosity solution when used on the whole domain with a convergent Newton's method solver \cite{Oberman2010a,nochetto2018two,DiscreteAlex2}. As pointed out in \cite{Froese13} the convergence of the hybrid discretization introduced in  \cite{Oberman2010b} is still an open problem and results with the hybrid discretization of  \cite{Oberman2010b} are comparable with the ones obtained with the filtered approach in \cite{Froese13}.

In this paper, we 
combine the analysis of the standard finite difference discretization given in \cite{Awanou-k-Hessian12}, the classical framework for convergence of monotone schemes to viscosity solutions and an argument first given in \cite{DiscreteAlex2} for the boundary values of uniform limits of discrete convex functions, to obtain the convergence to the viscosity solution of the hybrid discretization. We assume in this paper that the discrete problem has a solution which is close to the interpolant in the subdomain where the solution is smooth. 

With a monotone discretization one can transfer to the discrete level arguments for viscosity solutions for partial differential equations. But it does not allow to give, in general, results for the standard finite difference discretization for smooth solutions. 
In fact, the quadratic convergence rate of the latter 
was only known as `` formally second-order accurate''  \cite{Benamou2010}. 
Moreover, the theory of Barles and Souganidis \cite{Barles1991} cannot be applied directly to a hybrid discretization. 
The Banach fixed point theorem, which is ubiquitous in the finite element analysis of nonlinear problems, has been adapted to the Monge-Amp\`ere equation in \cite{Bohmer2008,Feng2009,Brenner2010b}. It was combined in \cite{Awanou-k-Hessian12} with the continuity of the eigenvalues of a matrix as a function of its entries to give an analysis of the
standard finite difference discretization for smooth solutions of the equation. 
A consequence of the continuity of the eigenvalues of a matrix is that, in the context of the standard finite difference discretization, the discrete Hessian of a mesh function  is positive definite near a strictly convex smooth solution. 
The impact of the results of this paper goes beyond the particular application considered. For example, the techniques used here may equally be applied to a hybrid scheme for the convex envelope presented in \cite{Oberman07}. 

The best results for the hybrid discretization are obtained when the subdomain where the solution is not smooth, the set of singular points, is known in advance. An adaptive mesh refinement scheme could make it easier to identify the set of singular points.
As with \cite{Oberman2010b}, one can take a conservative approach and include a priori in the set of singular points, points where either $f(x)$  is not H$\ddot{\text{o}}$lder continuous, $f(x)$ is too small or $f(x)$ is too large. It is very likely that the approximation will deteriorate at  points which are close to boundary points where $\partial \Omega$ is not $C^3$ or strictly convex and points close to boundary points where $g(x)$ cannot be extended to a $C^3$ function. They may be included in the set of singular points as well. The motivation to consider these points as singular points comes from the regularity theory of the Monge-Amp\`ere equation. See for example \cite[Theorem 1.1]{Trudinger08}.  

A standard finite difference discretization of the Dirichlet problem for the Monge-Amp\`ere equation was introduced in \cite{Dean2006}. Finite element discretizations have also been proposed, e.g. \cite{GlowinskiICIAM07,Bohmer2008,Feng2009,Brenner2010b,Lakkis11b,Davydov12,Glowinski2014}. As explained above, the performance of these methods can be explained with elliptic regularity. 
For other provably convergent schemes for the Monge-Amp\`ere equation, we refer to \cite{Mirebeau15,Feng2017,nochetto2018two}.

This paper is organized as follows. In the second section, we recall the notion of viscosity solution and present the hybrid discretization. 
We also define our notion of discrete convex function in the second section and the main notation of the paper. 
In the third section we rely on results on the analysis of discretizations of smooth solutions to motivate our assumptions on the existence of a discrete solution. 
In the fourth section we use the now classical arguments of \cite{Barles1991} and recent arguments given in \cite{Awanou-k-Hessian12,DiscreteAlex2} to prove the convergence  of the hybrid discretization to the viscosity solution.

\section{Viscosity solutions of the elliptic Monge-Amp\`ere equation and the hybrid discretization}
To avoid difficulties with a curved boundary, we assume in this paper that the domain $\Omega$ is rectangular. We further make the assumption that $\Omega=(0,1)^2 \subset $  $\R^2$. 
For given $f > 0$ continuous on $\tir{\Omega}$ and $g$ continuous on $\partial \Omega$, with a convex extension $\tilde{g} \in C(\tir{\Omega})$, we consider the Monge-Amp\`ere equation
\begin{align} \label{m1}
\begin{split}
\det D^2 u & = f \, \text{in} \, \Omega \\
u & = g \, \text{on} \, \partial \Omega.
\end{split}
\end{align}
 Let $0 < h \leq 1$ denote the mesh size. We assume without loss of generality  that  $1/h \in \mathbb{Z}$. Put
\begin{align*}
\mathbb{Z}_h & = \{x=(x_1,x_2)^T \in \R^2: x_i/h \in \mathbb{Z} \}\\
\Omega^h_0 &= \Omega \cap  \mathbb{Z}_h, \Omega^h = \tir{\Omega} \cap  \mathbb{Z}_h, \partial \Omega^h = \partial \Omega \cap \mathbb{Z}_h= \Omega^h \setminus \Omega^h_0.
\end{align*}
For $x \in \R^2$,  we denote the maximum norm of $x$ by $|x|=\max_{i=1,2} |x_i|$. The norm $| . |_{}$ is extended canonically to matrices. For an integer $j$, $|v|_{j,\Omega} = \text{sup}_{|\beta|=j} \text{sup}_{\Omega} |D^{\beta}v(x)|$ for a multi-index $\beta$. Let $\mathcal{M}(\Omega^h)$ denote the set of real valued functions defined on $\Omega^h$, i.e. the set of mesh functions. For a subset $T_h$ of $\Omega^h$, and $v^h \in \mathcal{M}(\Omega^h)$ 
we define
$$
|v^h|_{T_h} = \max_{x \in T_h} |v^h(x)|.
$$
The norm $| . |_{T_h }$ is extended canonically to matrix fields. Let $v$ be a continuous function on $\Omega$ and let $r_h (v)$ denote the unique element of $\mathcal{M}(\Omega^h)$ defined by
$$
r_h (v) (x) = v(x), x \in \Omega^h. 
$$
We extend the restriction operator $r_h$ canonically to vector fields and matrix fields. For a function $g$ defined on $\partial \Omega$, $r_h(g)$ defines the analogous restriction on $\partial \Omega^h$. We make the usual convention of denoting  constants by $C$ but will occasionally index some constants.

\subsection{Viscosity solutions} \label{visc}
A convex function $u \in C(\tir{\Omega})$ is a viscosity solution of \eqref{m1} if $u  = g \, \text{on} \, \partial \Omega$ and for all $\phi \in C^2(\Omega)$ the following holds
\begin{itemize}
\item[-] at each local maximum point $x_0$ of $u-\phi$, $f(x_0) \leq \det D^2 \phi(x_0)$
\item[-] at each local minimum point $x_0$ of $u-\phi$, $f(x_0) \geq \det D^2 \phi(x_0)$, if $D^2 \phi(x_0) \geq 0$, i.e. $D^2 \phi(x_0)$ has positive eigenvalues.
\end{itemize}
As explained in \cite{Ishii1990}, the requirement $D^2 \phi(x_0) \geq 0$ in the second condition above is natural for the two dimensional case we consider. 
The space of test functions in the definition above can be restricted to the space of strictly convex quadratic polynomials \cite[Remark 1.3.3]{Guti'errez2001}.

An upper semi-continuous convex function $u$ is said to be a viscosity sub solution of $\det D^2 u(x) = f(x)$ if the first condition holds and a lower semi-continuous convex function is said to be a viscosity super solution when the second holds. A  viscosity solution of \eqref{m1} is a continuous function which satisfies the boundary condition and is both a 
viscosity sub solution and a viscosity super solution.

Note that the notion of viscosity solution is a pointwise notion. It is not very difficult to prove that if $u$ 
is $C^2$ at $x_0$, then $u$ is a viscosity solution at the point $x_0$ of $\det D^2 u  = f $. 

For further reference, we recall the comparison principle of sub and super solutions, \cite[Theorem V. 2]{Ishii1990}. Let $u$ and $v$ be respectively sub and super solutions of $\det D^2 u(x)=f(x)$ in $\Omega$ and put
$$
u^* = \limsup_{y \to x, y \in \Omega} u(y) \, \text{and} \, v_* = \liminf_{y \to x, y \in \Omega} v(y).
$$
Then if $\sup_{x \in \partial \Omega} \max(u^*(x) - v_*(x),0)=M$, then $u(x) - v(x) \leq M$ in $\Omega$.

There are very few references which give an existence and uniqueness result for \eqref{m1} in the degenerate case $f \geq 0$. In \cite{Ishii1990} it is required that one can find a sub solution and a super solution. The difficulty is that the Monge-Amp\`ere equation is not often studied in convex but not necessarily strictly convex domains. Thus we assume in addition that $f >0$.
Since $f \in C(\tir{\Omega})$ it follows that there exists a constant $c_0 >0$ such that 
$$
f \geq c_0 > 0.
$$
We also assume that $g$ can be extended to a convex function $\tilde{g} \in C(\tir{\Omega})$. Then by \cite[Theorem 1.1]{Hartenstine2006}, \eqref{m1} has a unique Aleksandrov solution. The existence and uniqueness of a viscosity solution to \eqref{m1} in  $C(\tir{\Omega})$ then follows from the equivalence of viscosity and Aleksandrov solutions \cite[ Propositions 1.3.4 and 1.7.1]{Guti'errez2001}, under these assumptions. 



\subsection{A reformulation of convexity} \label{ref-conv}
We recall that a function $\phi \in C^2(\Omega)$ is convex on $\Omega$ if the Hessian matrix $D^2 \phi$ is positive semidefinite or $\lambda_1[\phi] \geq 0$ where $\lambda_1[\phi]$ denotes the smallest eigenvalue of $D^2 \phi$. This notion was extended to continuous functions in \cite{Oberman07}. See also the remarks on  \cite[p. 226 ]{Trudinger97b}. 
An upper semi-continuous function $u$ is convex in the viscosity sense if and only if it is a viscosity solution of $-\lambda_1[u] \leq 0$, that is, for all $\phi \in C^2(\Omega)$, whenever $x_0$ is a local maximum point of $u-\phi$, $-\lambda_1[\phi] \leq 0$. 
This can also be written $\text{max}(-\lambda_1[u] ,0) = 0 \, \text{in} \, \Omega$, c.f. \cite{Oberman07}.

The Dirichlet problem for the Monge-Amp\`ere equation \eqref{m1} can then be written
\begin{align} \label{m11}
\begin{split}
-\det D^2 u + f & = 0 \, \text{in} \, \Omega \\
\text{max}(-\lambda_1[u] ,0) &= 0 \, \text{in} \, \Omega, \\
\end{split}
\end{align}
with boundary condition $u=g$ on $\partial \Omega$.
We write \eqref{m11} as $F(u)=0$ and 
note that the form of the equation is chosen to be consistent with the definition of ellipticity used for example in \cite{Ishii1990}.

Since we have now rewritten in \eqref{m11} convexity as an additional equation, sub solutions and super solutions of $-\det D^2 u + f  = 0$ do not need to be convex. We have the following comparison principle for \eqref{m11} \cite[Example 2.1 and Corollary 7.1]{BardiDragoni}: let $u^*$ be an upper semi-continuous sub solution of $-\det D^2 u + f  = 0$ which is convex in the viscosity sense and let $u_*$ be a lower semi-continuous super solution of $-\det D^2 u + f  = 0$ (which is not necessarily convex). Then
\begin{equation} \label{comparison-principle}
\sup_{\Omega} (u^*-u_*) \leq \max_{\partial \Omega} (u^*-v_*).
\end{equation}
A viscosity solution of \eqref{m11} is also a viscosity solution as defined in section \ref{visc}, since an upper semi-continuous function which is convex in the viscosity sense is also convex \cite[Example 2.1 and Theorem 3.1]{BardiDragoni}.


\subsection{Standard finite difference discretization} \label{standard}

Let $\Omega_r$ be a bounded convex domain of $\R^2$ with piecewise linear boundary. Put
$$
\Omega_r^h = \tir{\Omega}_r \cap \mathbb{Z}_h.
$$
Let $e^i, i=1,2$ denote the $i$-th unit vector. We define first order difference operators acting on functions defined on $\mathbb{Z}_h$. 
For $x \in \mathbb{Z}_h$ 
\begin{align*}
\partial^i_{+} v^h(x) & \coloneqq \frac{v^h(x+he^i)-v^h(x)}{h} \\
\partial^i_{-} v^h(x) & \coloneqq \frac{v^h(x)-v^h(x-he^i)} {h}\\
\partial^i_h v^h(x) & \coloneqq \frac{v^h(x+he^i)-v^h(x-he^i)}{2 h}.
\end{align*}
Note that
\begin{align}
\partial^i_{+} \partial^i_{-} v^h(x) & = \frac{v^h(x+he^i)-2v^h(x)+v^h(x-he^i)}{h^2} \label{second-disc1}
\end{align}
\begin{align}
\begin{split}
\partial^i_h \partial^j_h v^h(x) & = \frac{1}{4 h^2} \bigg\{v^h(x+he^i+h e^j)+v^h(x-he^i-h e^j) \\
& \qquad \qquad \qquad -v^h(x+he^i-h e^j)-v^h(x-he^i+ he^j)\bigg\}, i \neq j. \label{second-disc2}
\end{split}
\end{align}
The discrete Hessian is defined by
$$
\mathcal{H}_d(v^h) \coloneqq (\mathcal{H}_d(v^h))_{i,j=1,2}, (\mathcal{H}_d(v^h))_{ii} = \partial^i_{+} \partial^i_{-} v^h \text{ and } (\mathcal{H}_d(v^h))_{ij} = \partial^i_h \partial^j_h v^h, i \neq j.
$$
Put
\begin{align*}
\Omega_{r,0}^h  = \{ \, x \in \Omega_r^h, \mathcal{H}_d(v^h)(x) \,\text{ is defined for} \, v^h \in \mathcal{M}(\Omega^h) \, \} \text{ and }
\partial  \Omega_r^h  =  \Omega_r^h \setminus \Omega_{r,0}^h.
\end{align*}
We define
$$
M_r [v^h]   \coloneqq \det \mathcal{H}_d(v^h).
$$
The discrete version of \eqref{m1} on $\Omega_r$ takes the form
\begin{align} \label{disc-smooth-domain}
- M_r [u^h_r]  + r_h(f) =0 \, \text{in} \, \Omega_{r,0}^h, u^h_r =r_h(u) \, \text{on} \, \partial  \Omega_r^h.
\end{align}

Higher order finite difference operators are obtained by combining the above difference operators. For a multi-index $\beta=(\beta_1,\beta_2) \in \mathbb{N}^2$, we define
$$
\partial^{\beta}_{+} v^h \coloneqq  \partial^{\beta_1}_{+}  \partial^{\beta_2}_{+}v^h.
$$
The operators $\partial^{\beta}_{-}$ and $\partial^{\beta}_{h}$ are defined similarly. For a matrix $A$, we recall that the cofactor matrix $\cof A$ is defined by $(\cof A)_{ij}=(-1)^{i+j} d(A)_i^j$ where $d(A)_i^j$ is the determinant of the matrix obtained from $A$ by deleting the $i$th row and the $j$th column. Let $\mathcal{M}(\Omega_r^h)$ denote the set of real valued functions defined on $\Omega_r^h$ 
and let $L_h$ denote a discrete uniformly elliptic linear operator 
\begin{align*}
L_h v^h(x) = \sum_{i,j=1}^2 a^{ij}(x) \partial^i_- \partial^j_+ v^h(x), 
x \in \Omega_{r,0}^h,
\end{align*}
i.e. the matrix $(a^{ij}(x))_{i,j=1,2}$ is uniformly positive definite. 
We now define discrete analogues of the H$\ddot{\text{o}}$lder norms and semi-norms following \cite{Johnson74}. Let $[\xi,\eta]$ denote the set of points $\zeta \in \Omega_r^h$ such that $\xi_j \leq \zeta_j \leq \eta_j, j=1,2$. Then for $v^h \in \mathcal{M}(\Omega_r^h), 0 < \alpha < 1$, we define
\begin{align*}
|v^h|_{j,\Omega_{r,0}^h} & = \, \text{max} \, \{\, |\partial^{\beta}_{+}v^h (\xi)|, |\beta|=j, [\xi,\xi+\beta] \subset \Omega_r^h \, \}\\
[v^h]_{j,\alpha,\Omega_{r,0}^h} & = \, \text{max} \, \bigg\{\, \frac{|\partial^{\beta}_{+}v^h (\xi)- \partial^{\beta}_{+}v^h (\eta)|}{( |\xi-\eta|)^{\alpha}}, 
|\beta|=j, \xi \neq \eta, [\xi,\xi+\beta] \cup [\eta,\eta+\beta] \subset \Omega_r^h  \, \bigg\}\\
||v^h||_{p,\Omega_{r,0}^h} & = \, \text{max}_{j \leq p} \, |v^h|_{j,\Omega_{r,0}^h} \text{ and }
||v^h||_{p,\alpha,\Omega_{r,0}^h}  = ||v^h||_{p,\Omega_{r,0}^h}+ [v^h]_{p,\alpha,\Omega_{r,0}^h}. 
\end{align*}
The above norms are extended canonically to vector fields and matrix fields by taking the maximum over all components. For $j=0$, we have discrete analogues of the maximum and $C^{0,\alpha}$ norms and in the former case, we will also use $|v^h|_{\Omega_{r,0}^h}$ at the place of $|v^h|_{0,\Omega_{r,0}^h}$. We have \cite[Lemma 3.4]{Thomee1970}
\begin{thm} \label{discShauderPoisson}
Assume $ 0<\alpha<1$ and $v^h=0$ on $\partial \Omega_r^h$. Then there are constants $C$ and $h_0$ such that for $v^h \in \mathcal{M}(\Omega_r^h), h \leq h_0$
\begin{align} \label{discShauderPoisson0}
||v^h||_{2,\alpha,\Omega_{r,0}^h} \leq C ||L_h \, v^h||_{0,\alpha,\Omega_{r,0}^h}, 
\end{align}
with the constant $C$ independent of $h$.
\smallskip
\end{thm}

It can be shown that the constant $h_0$ depends only on $c_0$ and $m$ such that $c_0 \leq f \leq m/2$. 

We recall from  \cite{Awanou-k-Hessian12} that if 
$v$ is a strictly convex function and $D^2 v$ has smallest eigenvalue uniformly bounded below by a constant $a >0$, then for $\eta=a/4$, we have $w$ strictly convex, whenever $||w-v||_{C^2(\Omega)} < \eta$. Moreover the smallest eigenvalue of $D^2 w$ is uniformly bounded below by $3 a/2$. 
Since for $v \in C^{4}(\Omega)$,
\begin{align} \label{consistent3}
|r_h(D^2 v) - \mathcal{H}_d( r_h(v) )|_{\Omega_{r,0}^h} \leq C h^2 |v|_{4,\Omega},
\end{align}
there exists $0<h_0\leq h_1$ such that $h\leq h_1$, $\mathcal{H}_d( r_h(u) )$ has  smallest eigenvalue uniformly bounded below by $3 a/2$. 

We now summarize the approach in \cite{Awanou-k-Hessian12} for the solvability of \eqref{disc-smooth-domain}. Define
\begin{equation} \label{ball-h}
B_{\rho} ( u) = \{v^h \in \mathcal{M}(\Omega^h), ||v^h- r_h(u) ||_{2,\alpha,\Omega_{r,0}^h} \leq \rho \},
\end{equation}
and the operator
$R^h: \mathcal{M}(\Omega^h) \to \mathcal{M}(\Omega^h)$ by
\begin{align*} 
\bigg(  \cof (\mathcal{H}_d \, r_h u) \bigg) : \mathcal{H}_d  (v^h-R^h v^h) &=  \det (\mathcal{H}_d \, v^h)  - f  \, \text{in} \, \Omega^h_{r,0} \\
R^h(v^h) & = r_h(u)  \, \text{on} \, \partial \Omega_r^h,
\end{align*}

It is proved in \cite{Awanou-k-Hessian12}, using the discrete Schauder estimates of Theorem \ref{discShauderPoisson} and the above continuity of the eigenvalues of a matrix, that for $\rho=$O$(h^2)$ and $h$ sufficiently small, $R_h$ maps $B_{\rho} ( u)$ into itself and is a strict contraction in $B_{\rho} ( u)$. 

\subsection{Monotone schemes} \label{mono}

Let us denote by $F_h(v^h) \equiv \hat{F}_h(v^h(x), v^h(y)|_{ y \neq x})$ a discretization of $F(v)$. We recall the elements of the convergence theory of Barles and Souganidis \cite{Barles1991} and how its conditions were met by the discretization introduced in \cite{Oberman2010a}. Let $\Omega^h_s$ denote a subset of $\Omega^h$ and let $\partial \Omega^h_s$ denote its boundary, i.e. $\partial \Omega^h_s = \Omega^h \setminus \Omega^h_s$.

The scheme $F_h(v^h) =0$ is said to be monotone if for $z^h$ and $w^h$ in  $\mathcal{M}(\Omega^h_s)$, $z^h(y) \geq w^h(y), y \neq x$ implies $\hat{F}_h(z^h(x), z^h(y)|_{ y \neq x}) \geq \hat{F}_h(z^h(x), w^h(y)|_{ y \neq x})$. Here we use the partial ordering of $\R^2$, $(a_1,b_1) \geq (a_2,b_2)$ if and only if $a_1 \geq a_2$ and $b_1 \geq b_2$.

The scheme is said to be consistent if for all $C^2$ functions $\phi$, and a sequence $x_h \to x \in \Omega$,
$\lim_{h \to 0} F_h (r_h(\phi)) (x_h) = F(\phi)(x)$. 

Finally the scheme is said to be stable if $F_h(v^h)=0$ has a solution $v^h$ which is bounded independently of $h$.

It follows from \cite{Barles1991,nochetto2018two,DiscreteAlex2} that a consistent, stable and monotone scheme has a solution $v^h$ which converges locally uniformly to the unique viscosity solution of \eqref{m11}. Note that the convexity assumption on the exact solution is enforced through the definition of $F(v)$.


We recall the expression of the consistent, monotone and stable discretization of $\lambda_1[z]$ introduced in \cite{Oberman07}. For simplicity we consider only wide stencils. 
We have at an interior grid point $x$
\begin{equation} \label{disc-eig}
\lambda_1^h[v^h](x) = \min_{\alpha^h \in \R^2} \frac{v^h(x+ \alpha^h) -2 v^h(x) + v^h(x-\alpha^h)}{|\alpha^h|^2},
\end{equation}
where by $\alpha^h \in \R^2$ we mean vectors $\alpha^h$ for which the above expression is well defined for grid points. 

We also recall the expression $M_s[v^h]$ of the discretization of $\det D^2 v$ used in \cite{Oberman2010a}. For $x \in \Omega^h_s$ we denote by $W_h(x)$ the set of orthogonal bases of $\R^2$ such that for $(\alpha_1,\alpha_2) \in W_h(x)$ $x\pm \alpha_i \in \Omega^h, \forall i$. We have
\begin{equation} \label{disc-det}
M_s[v^h] (x) = \inf_{(\alpha_1,\alpha_2) \in W_h(x)} \prod_{i=1}^2 \frac{v^h(x+ \alpha_i) -2 v^h(x) + v^h(x-\alpha_i)}{|\alpha_i|^2}.
\end{equation}



The monotone discretization of \eqref{m11} can then be written
\begin{align} \label{m11mh}
\begin{split}
-M_s[u^h_s](x) + r_h(f)(x) & =0, x \in \Omega^h_s \\
\text{max}(-\lambda_1^h[u^h_s](x) ,0) &= 0, x \in \Omega^h_s \\
u^h_s(x) & = r_h(u)(x)\, \text{on} \, \partial \Omega^h_s. 
\end{split}
\end{align}
As with \cite{Oberman2010a}, the first two equations of \eqref{m11mh} are combined in a single equation. Recall that $x^+ = \max(x,0)$ and define
\begin{equation*} 
M_s^+[v^h] (x)= \inf_{(\alpha_1,\alpha_2) \in W_h(x)} \prod_{i=1}^2 \max \bigg( \frac{v^h(x+ \alpha_i) -2 v^h(x) + v^h(x-\alpha_i)}{|\alpha_i|^2},0 \bigg).
\end{equation*}
Then \eqref{m11mh} can be written
\begin{align} \label{m11mh2}
\begin{split}
-M_s^+[u^h_s](x) + r_h(f)(x) & =0, x \in \Omega^h_s \\
u^h_s(x) & = r_h(u)(x)\, \text{on} \, \partial \Omega^h_s. 
\end{split}
\end{align}
It is known \cite{Aw-Matamba,Froese2018,nochetto2018two} that \eqref{m11mh2} has a solution.

Finally we recall the following discrete comparison principle, the proof of which follows from arguments given in \cite{nochetto2018two}. If $M_s^+[v^h] \leq M_s^+[w^h]$ in $\Omega^h_s$ and $v^h \geq w^h$ on $\partial \Omega^h_s$, then $v^h \geq w^h$ on $\Omega^h_s$.

\begin{rem} We note that for the implementation of \eqref{m11mh2}, linear interpolation near the boundary can be used. Since the operator $M_s^+$ is pointwise consistent and convergence to viscosity solution is point by point, we do not discuss linear interpolation. It can be shown that for smooth solutions \cite{Awanou-symmetry}, it leads to a convergence rate O($h+ \ud \theta$) where $\ud \theta$ is called directional resolution and $\ud \theta \to 0$ as $h \to 0$.

\end{rem}

\subsection{The hybrid discretization} \label{hybrid}

\medskip

\begin{defn}
We call a point $x \in \Omega$ a regular point if the solution $u$ of \eqref{m1} is $C^2$ in a neighborhood of $x$. A point which is not a regular point is called a singular point.
\end{defn}

The above definition is natural if one considers the one dimensional Monge-Amp\`ere equation $-u''(x)=f$ and a standard finite difference approximation. In particular, at a regular point $x$, by a Taylor series expansion,
\begin{align} \label{consistent2}
\begin{split}
\lim_{h \to 0 }\max_{i,j=1,\ldots,n} |\partial^2 v(x)/(\partial x_i \partial x_j) - \partial^j_- \partial^i_+ (r_h v) (x)|   =0.  
\end{split}
\end{align}
Next, for $v \in C^4(\Omega)$, and $x \in \Omega$
\begin{align} \label{consistent}
\begin{split}
\max_{i,j=1,2 } |\partial^2 v(x)/(\partial x_i \partial x_j) - \partial^j_- \partial^i_+ (r_h v) (x) | \leq C h^2 |v|_{4,\Omega}.
\end{split}
\end{align}

Let $\Omega_r$ denote an open (rectangular) subset of $\Omega$ such that at every point $x$ of $\Omega_r$ the exact solution $u$ is $C^2$ in a neighborhood of $x$. Using the notation of section \ref{standard} we define
$$
\Omega_s^h = \Omega_0^h \setminus \Omega_{r,0}^h.
$$


\begin{defn}
By a discrete convex function, we mean a mesh function $v^h$ such that 
\begin{align}\label{disc-convex}
\begin{split}
\lambda_1^h[v^h] & \geq 0 \, \text{in} \, \Omega_s^h \\
\mathcal{H}_d \, v^h & \geq 0 \, \text{in} \, \Omega_{r,0}^h.
\end{split}
\end{align}
Strictly discrete convex functions are defined analogously.
\end{defn}

We note that a discrete convex function in the sense of the above definition is not necessarily convex on $\Omega_0^h$. See \cite{Oberman07} for the case $\Omega_s^h=\Omega^h_0$ and \cite{Moriguchi12} for a counterexample showing that the discrete Hessian $\mathcal{H}_d \, v^h$ can be positive without the mesh function $v^h$ being convex in the usual sense. The minor abuse of terminology we make is justified by  Theorem \ref{main} below which says in particular that the uniform limit of mesh functions which satisfy \eqref{disc-convex} and solve the discrete Monge-Amp\`ere equation \eqref{hybrid-disc} below, is convex.

For a subset $T_h$ of $\Omega^h$, we denote by $\mathcal{C}^h(T_h)$ the cone of discrete convex functions on $T_h$ and by $\mathcal{C}^h_0(T_h)$ the cone of strictly discrete convex functions on $T_h$. We define on $\Omega_0^h$ for a mesh function $v^h$, $F_h(v^h)$ by
\begin{align} \label{F-hybrid}
\begin{split}
F_h(v^h) (x) & = -M_s^+[v^h](x) + r_h(f)(x), x \in  \Omega_s^h \\ 
F_h(v^h) (x) &= -M_r[v^h](x) + r_h(f)(x), x \in \Omega_{r,0}^h.
\end{split}
\end{align}
Put $\mathcal{C}^h = \mathcal{C}^h(\Omega^h_0)$. The hybrid discretization of \eqref{m11} can then be written: find $u^h \in \mathcal{C}^h$
\begin{align} \label{hybrid-disc}
F_h (u^h)(x)=0 \, \text{in} \,   \Omega^h_0, u^h(x) = r_h(g)(x)\, \text{on} \, \partial \Omega^h. 
\end{align}

In \cite{Oberman2010b}, the authors use a weight function to write the hybrid discretization as a combination of the monotone scheme and the standard finite difference discretization. We omit it in this paper as it plays no role in our analysis.

\section{Assumptions on the existence  of a discrete convex solution} 

Were the solution known on $\partial \Omega^h_r$, we could use the arguments of section \ref{standard} to find the solution $u_h$ of the hybrid discretization in the smooth domain. We could then use the arguments of section \ref{mono} to get the solution in the domain where the solution is not smooth. However, in this approach the discrete problem would not be satisfied at mesh points of $\partial \Omega^h_r$ which are interior points. This suggests that Problem \eqref{F-hybrid} has a solution $u_h$ in  the ball
$$
B_{} (r_h(u)) = \{ \, v^h \in \mathcal{M}(\Omega^h), |v^h-r_h(u)|_{\Omega_{r}^h} \leq C h^2, v^h=r_h(g) \, \text{on} \, \partial \Omega^h\},
$$
for a constant $C$ and for $h$ sufficiently small.

We note that uniqueness of a discrete solution is important for the use of Newton's method, but not necessary for the proof of convergence of the discretization. 

\section{Convergence to the viscosity solution of the hybrid discretization} \label{cvg-hybrid}
We recall that we follow the usual convention of denoting  constants by $C$ but will occasionally index some constants.

We first prove the stability of the hybrid discretization \eqref{hybrid-disc}. Then we prove that the half-relaxed limits
\begin{align*}
u^*(x) = \limsup_{y \to x, h \to 0} u^h(y) & = \lim_{\delta \to 0} \sup\{ \,u^h(y), y \in \Omega_0^h, |y-x| \leq \delta, 0<h\leq \delta \, \} \\
 u_*(x) = \liminf_{y \to x, h \to 0} u^h(y) & = \lim_{\delta \to 0} \inf \{ \,u^h(y), y \in \Omega_0^h, |y-x| \leq \delta, 0<h\leq \delta \, \},
\end{align*}
are respectively sub and super solutions of \eqref{m11}. In addition we show that on $\partial \Omega$ $u^* \leq g \leq u_*$ using a result given in \cite{DiscreteAlex2}.

\subsection{Stability on the set of regular points}
Since $ u_h \in B_{} (r_h(u))$ we have for $h$ sufficiently small,
\begin{align}  \label{partial-bound2}
|u^h|_{\Omega_{r}^h} & \leq |u^h- r_h(u)|_{\Omega_{r}^h} + | r_h(u)|_{\Omega_{r}^h}   \leq C h^2 + | r_h(u)|_{\Omega_{r}^h}   \leq C,
\end{align}
since by definition $u \in C(\tir{\Omega})$.

\subsection{Stability on the set of singular points}
It follows from our definition of $\Omega^h_s$ in section \ref{hybrid} that $\partial \Omega^h_s \cap \partial \Omega^h \neq \emptyset$ and $\partial \Omega^h_s \cap \partial \Omega^h_r \neq \emptyset$. Since $u^h = g$ on $\partial \Omega^h$ and $u^h$ bounded on $\partial \Omega^h_r$ by \eqref{partial-bound2}, we conclude that $u^h$ is bounded on $\partial \Omega^h_s$. 

Since $\Omega$ is bounded, there exists $A>0$ such that $|x| \leq A$ for $x \in \Omega$. Assume that $0\leq f \leq m$. We compare $u^h_{}$ with $u_m=\sqrt{m}/2 |x|^2 - \sqrt{m}/2 A^2-C$ for a large positive constant $C$ such that $u^h \geq u_m$ on $\partial \Omega^h_s$. We have $M_s^+[u_m]=m$ and 
$$
M_s^+[u^h_{}] = f \leq m = M_s^+[u_m].
$$
By the discrete comparison principle, we obtain $u^h_{} \geq u_m \geq - \sqrt{m}/2 A^2-C$. 
By discrete convexity, see for example \cite{DiscreteAlex2}, we have 
$$
u^h_{} \leq \max_{x \in \partial \Omega^h_s} u^h(x).
$$
This proves that
\begin{align} \label{partial-bound}
|u^h|_{\Omega_s^h} & \leq  C. 
\end{align}
Inequalities \eqref{partial-bound} and \eqref{partial-bound2} allow us to state the following theorem
\begin{thm} \label{u-bound}
There is a constant $C>0$ independent of $h$ such that for $h$ sufficiently small, the solution $u^h$ of \eqref{hybrid-disc} satisfies $|u^h|_{\Omega^h} \leq C$.
\end{thm}

\subsection{Sub and super solution property of the half-relaxed limits }
Theorem \ref{u-bound} implies that the half-relaxed limits are well defined. We have 
\begin{thm} \label{singular}
The upper half-relaxed limit $u^*$ is a viscosity sub solution of $\det D^2 u(x)=f(x)$  and the lower half-relaxed limit $u_*$ is a viscosity super solution of $\det D^2 u(x)=f(x)$ at every  point of $\Omega \setminus \Omega_r$. 
In addition, $u^*$ is a viscosity solution of $-\lambda_1[u](x) \leq 0$ at every point of $\Omega \setminus \Omega_r$. 
\end{thm}
\begin{proof}
The result follows from the results of \cite{Barles1991} and the stability, consistency and monotonicity of the scheme used in the ''singular'' part of the domain. For the convenience of the reader, we give a proof following \cite{Bouchard2009}.

We show that $u_*$ is a viscosity super solution of $\det D^2 u(x)=f(x)$ at every  point of $\Omega \setminus \Omega_r$. 

It follows from the definitions that $u_*$ is lower semi-continuous. Let $x_0 \in \Omega$ and $\phi \in C^2(\Omega)$ with $D^2 \phi(x_0) \geq 0$ such that $u_*-\phi$ has a local minimum at $x_0$ with $(u_*-\phi)(x_0)=0$. Without loss of generality, we may assume that $x_0$ is a strict local minimum.

Let $B_0$ denote a closed ball contained in $\Omega$ and containing $x_0$ in its interior. We let $x_l$ be a sequence in $B_0$ such that $x_l \to x_0$ and $u^{h_l}(x_l) \to u_*(x_0)$ and let $x'_l$ be defined by
$$
c_l  \coloneqq  (u^{h_l} - \phi)(x'_l) = \min_{B_0} u^{h_l} - \phi.
$$
Since the sequence $x'_l$ is bounded, it converges to some $x_1$ after possibly passing to a subsequence. Since $(u^{h_l}-\phi)(x'_l) \leq (u^{h_l}-\phi)(x_l)$ we have
$$
(u_*-\phi)(x_0) = \lim_{l \to \infty} (u^{h_l}-\phi)(x_l) \geq \liminf_{l \to \infty} (u^{h_l}-\phi)(x'_l) \geq (u_*-\phi)(x_1).
$$
Since $x_0$ is a strict minimizer of the difference $u_*-\phi$, we conclude that $x_0=x_1$ and $c_l \to 0$ as $l \to \infty$.

By definition
$$
u^h(x) \geq \phi(x) + c_l, \forall x \in B_0,
$$
and thus, by the monotonicity of the scheme
$$
0 = \hat{F}_h(u^h(x_0), u^h(y)|_{y \neq x_0}) \geq \hat{F}_h(u^h(x_0), (\phi(y) +c_l)|_{y \neq x_0})=\hat{F}_h(\phi(x_0), (\phi(y) +c_l)|_{y \neq x_0}), 
$$
which gives by the consistency of the scheme $\det D^2 \phi(x_0) - f(x_0) \leq 0$.

Similarly one shows that if $\phi \in C^2(\Omega)$ and $u^*-\phi$ has a local maximum at $x_0$ with $(u^*-\phi)(x_0)=0$, we have $\det D^2 \phi(x_0) - f(x_0) \geq 0$ and $-\lambda_1[\phi](x_0) \leq 0$. 
\end{proof}

For the behavior at regular points, we have
\begin{thm} \label{regular}
At every regular point $x \in \Omega$,
$$
u_*(x)=u^*(x)=u(x).
$$
And thus $u_*$ and $u^*$ are viscosity solutions of \eqref{m11} at $x \in \Omega_r$.
\end{thm}
\begin{proof}
Since $u^h \in B_{} (r_h(u))$ by assumption, $u^h$ converges to $u$ uniformly on compact subsets of $\Omega_r$. The result then follows since $u$ is $C^2$ at $x \in \Omega_r$ and hence a viscosity solution at $x \in \Omega_r$.
\end{proof}

\begin{thm} \label{boundary}
On $\partial \Omega$ $u^* \leq g \leq u_*$.
\end{thm}
\begin{proof} 
We make an essential use of \cite[Theorem 4.1]{DiscreteAlex2}. First, since by assumption $u^h \in B_{} (r_h(u))$ we have on $\tir{\Omega_r} \cap \partial \Omega$, $u=u^*  = u_*=g$. Next, we note that $\Omega^h_s \cup \partial \Omega^h_s= \big( \tir{\Omega} \setminus \Omega_r\big) \cap \mathbb{Z}_h$.  And $u^h$ is uniformly bounded on $\Omega^h_s$ and discrete convex. By \cite[Theorem 4.1]{DiscreteAlex2}, there exists a subsequence $u^{h_k}$ which converges uniformly on compact subsets of $\tir{\Omega} \setminus \Omega_r$ to a convex function $v$ which satisfies $v=g$ on $(\tir{\Omega} \setminus \Omega_r) \cap \partial \Omega$ and which is continuous up to the boundary. By definition, we have $v=u^*  = u_*$ on $\Omega$ and hence for $\zeta \in \partial \Omega$, $\lim_{x \to \zeta} u_*(x) \geq g(\zeta)$ and $\lim_{x \to \zeta} u^*(x) \leq g(\zeta)$.
\end{proof}

We close this section by stating the main result of this paper 
\begin{thm} \label{main}
The solution $u^h$ of \eqref{hybrid-disc} converges uniformly on compact subsets to the unique solution of \eqref{m11}.
\end{thm}
\begin{proof}
Using the definitions we have $u_* \leq u^*$ on $\tir{\Omega}$. By Theorem \ref{boundary} $u^* \leq g \leq u_*$ on $\partial \Omega$. We recall that $u^*$ is convex in the viscosity sense and hence convex. By the comparison principle \eqref{comparison-principle} and Theorems \ref{singular} and \ref{regular}, we have $u_* \geq u^*$ on $\Omega$. Hence $u_*=u^*$ on $\tir{\Omega}$ and both $u_*$ and $u^*$ are thus continuous on $\tir{\Omega}$. 
We conclude that
$u_*=u^*$ is the unique viscosity solution of \eqref{m11}. By uniqueness of the viscosity solution $u_*=u^*=u$ and hence $u^h$ converges uniformly on compact subsets to $u$ by  \cite[Lemma 1.9 p. 290]{Bardi97}. 
\end{proof}

\section*{Acknowledgements}
The author would like to thank the referees for a careful reading of the paper and for suggestions which help improved the presentation of the paper. The author was partially supported by NSF DMS grant No 1319640 and NSF DMS grant \# 1720276.


\end{document}